\documentclass[%
 aip,
 amsmath,amssymb,
reprint, onecolumn 
]{revtex4-1}
\usepackage{amsthm}%
\usepackage{tabularx}
\usepackage{amssymb}

\usepackage{graphicx}
\usepackage{dcolumn}
\usepackage{bm}

\usepackage[utf8]{inputenc}
\usepackage[T1]{fontenc}
\usepackage{mathptmx}
\usepackage{etoolbox}

\newtheorem{thm}{Theorem}[section]

\theoremstyle{remark}

\theoremstyle{definition}

\DeclareMathAlphabet{\mathpzc}{OT1}{pzc}{m}{it}

\def\R{\mathbb{ R}}

\newcommand{\be}{\begin{equation}}
\newcommand{\ee}{\end{equation}}

\newcommand{\ba}{\begin{array}}
\newcommand{\ea}{\end{array}}

\newcommand{\bg}{\begin{gathered}}
\newcommand{\eg}{\end{gathered}}

\renewcommand{\l}{\lambda}
\newcommand{\f}{\phi}

\newcommand{\bea}{\begin{eqnarray}}
\newcommand{\eea}{\end{eqnarray}}

\newcommand{\Sum}{\sum_{n=0}^\infty}

\makeatletter
\def\@email#1#2{%
 \endgroup
 \patchcmd{\titleblock@produce}
  {\frontmatter@RRAPformat}
  {\frontmatter@RRAPformat{\produce@RRAP{*#1\href{mailto:#2}{#2}}}\frontmatter@RRAPformat}
  {}{}
}%
\makeatother

\usepackage{tikz}
\usetikzlibrary{intersections,calc}
\usepackage{tikz}

\begin{document}

\preprint{AIP/123-QED}

\title{On the Meixner–Pollaczek polynomials and the Sturm--Liouville problems}
\medskip
\medskip
\author{Mourad E. H.  Ismail$^{*}$ and Nasser Saad}
\affiliation{${}^*$Department of Mathematics, University of Louisiana at Lafayette, 
Lafayette, LA 70504-3568 USA\\ 
${}^1$School of Mathematical and Computational Sciences,
University of Prince Edward Island, 550 University Avenue,
Charlottetown, PEI, Canada C1A 4P3..}

\begin{abstract}
This work provides a detailed study of Meixner–Pollaczek polynomials and employs the central difference operator to study the Sturm--Liouville problem. It presents two linearly independent solutions to the recursion relation, along with the associated difference equations. Additionally, the establishment of second-kind functions is discussed.
\end{abstract}



\maketitle
 \section{Introduction}\label{Sec:Into}
\noindent The Meixner–Pollaczek polynomials are a family of orthogonal polynomials.  They were first introduced by Meixner \cite{m1934} and later rediscovered and generalized by Pollaczek \cite{Poll50}, see also \cite{a1985,c1978,e1953,i2005,k2010,k1989}.  They have, for $0<\phi<\pi$, $\lambda>0$ and $x\in(-\infty,\infty)$, the hypergeometric representation
 \begin{equation}\label{eq1}
P_n^{(\lambda)}(x;\phi)=\frac{(2\lambda)_n}{n!}e^{i\,n\,\phi}{}_2F_1\left(\begin{array}{cc}-n,&\lambda+i\, x\\ 2\lambda\end{array}\bigg|1-e^{-2i\,\phi}\right),\quad i=\sqrt{-1},~n=0,1,2,\cdots,
\end{equation}
and 
\begin{equation}\label{eq2}
P_n^{(\lambda)}(x;\phi)= e^{i n \phi } \sum _{k=0}^n \frac{(\lambda+i x)_k (\lambda -i x)_{n-k}}{k! (n-k)!}e^{-2i k \phi}
\end{equation}
where $(\alpha)_k$ is the Pochhammer symbol defined by $(\alpha)_0=1, (\alpha)_n=\alpha(\alpha+1)\cdots(\alpha+n-1)$, $(-n)_n=(-1)^n\, n!$ and $(-n)_k=0$ for $n< k$.
They have the generating function
\begin{equation}\label{eq3}
\sum_{n=0}^\infty P_n^{(\lambda)}(x;\phi)t^n= (1-t e^{i \phi})^{-(\lambda-i x)} (1-t e^{-i \phi})^{-(\lambda+i x)}. 
\end{equation} 
They are orthogonal on the real line with respect to the weight function
\begin{equation}\label{eq4}
\omega(x,\lambda,\phi)=e^{(2\phi-\pi)\,x} |\Gamma(\lambda+ix)|^2
\end{equation}
with the orthogonality relation
\begin{align}\label{eqort}
\int_{-\infty}^\infty {P}_n^{(\lambda)}(x;\phi){P}_m^{(\lambda)}(x;\phi) \,  \omega(x,\lambda,\phi) dx = 
\frac{2\,\pi\, \Gamma(n+2\lambda)}{(2\sin\phi)^{2\lambda} n!}\,\delta_{mn}
\end{align}
where the Kronecker delta $\delta_{mn}=1$ if $m=n$, otherwise it is zero.
\vskip0.1true in
\noindent The raising and lowering relations of the Meixner--Pollaczek polynomials involve the   the central difference operator   $T$  defined by \cite[p. 436]{o2010} 
 \begin{equation}\label{t1}
 (Tf)(x) = \frac{f(x+i/2)-f(x-i/2)}{i}
 \end{equation}
 and acting on a space of polynomials and extended by linearity to the closure of the space of polynomials.
 If we define the polynomials $\{ \phi_n^{(\lambda)}(x) \}$ by
\begin{equation}
 \phi_n^{(\lambda)}(x) = (\lambda + (1-n)/2 + ix)_n,\quad n=0,1,2,\cdots,
\end{equation}
we see that in the $T$-calculus, these polynomials have a similar role as $\{x^n\}$ in differential calculus.
 Indeed 
\begin{equation} \label{eqTf}
(T\phi_n^{(\lambda)})(x) = i \,n\, \phi^{(\lambda)}_{n-1}(x)
\end{equation}
and for $k\leq n$
\begin{equation} \label{eqtk}
(T^k\phi^{(\lambda)}_n)(x) =(-i)^k\, (-n)_k\, \f^{(\l)}_{n-k}(x).
\end{equation}
Using equation \eqref{eq1}, it is straightforward for verify the lowering relation
\begin{align}\label{o3}
{T P_n^{(\lambda)}(x;\phi)=2\,\sin\phi\, P_{n-1}^{(\lambda+1/2)}(x;\phi)}
\end{align}
and in general
\begin{align}\label{o5}
{T^k P_n^{(\lambda)}(x;\phi)=(2\,\sin\phi)^k\, P_{n-k}^{(\lambda+k/2)}(x;\phi).}
\end{align}
Further by means of the weight function  \eqref{eq4}, it is not difficult to show that the raising relation is 
\bea
\label{eqraise}
T[\omega(x;\l,\f)\, P_n^{(\lambda)}(x;\phi)] = -(n+1)\, \omega(x;\l-1/2,\f)\, P_{n+1}^{(\lambda-1/2)}(x;\phi).
\eea
It is noteworthy that if we substitute $t$ with $ct$ in the generating functions presented in \eqref{eq3}, and subsequently select  $ce^{i\phi} = e^{i\theta}$ and $ce^{-i\phi} = e^{i\psi}$, we derive the generating function
\begin{equation}\label{eq13}
\sum_{n=0}^\infty P_n^{(\lambda)}(x;\theta,\psi)\, t^n= (1-t e^{i\theta})^{-(\lambda-i x)} (1-t e^{i \psi})^{-(\lambda+i x)},
\end{equation} 
for the generalized Meixner-Pollaczek polynomials \cite{a2013}
 \begin{equation}\label{eq14}
P_n^{(\lambda)}(x;\theta,\psi)=\frac{(2\lambda)_n}{n!}e^{i n \theta}{}_2F_1\left(\begin{array}{cc}-n,&\lambda+i\, x\\ 2\lambda\end{array}\bigg|
1-e^{i (\psi-\theta)}\right),~n=0,1,2,\cdots.
\end{equation}
which is essentially equivalent to equations \eqref{eq3} and \eqref{eq1}, respectivelly.
\vskip0.1true in
\noindent  In \S 2 we introduce the analogue of $e^{ixt}$ for the $T$ calculus, that is a function 
$E^{\lambda}(x, t)$ which satisfies $TE^{\lambda}(x, t)= it E^{\lambda}(x, t)$. 
 In \S 3 we expand the function  $E^{\lambda}(x, t)$ into a series of Meixner--Polllaczek polynomials, see \eqref{eqPW}. This is the analogue of the expansion of a plane wave $e^{ixt}$ in spherical harmonics. Section 4 investigates the Sturm--Liouville theory of the second-order operator 
 $[1/w(x)[T[pT]$. In Section 5, we determine the large degree asymptotics of 
 the Meixner--Pollaczek polynomials and the corresponding numerator polynomials \cite[Section 2.3]{i2005}. We then give a new derivation of the orthogonality relation 
 \eqref{eqort}.   This technique modifies the technique in \cite{Poll56} to an unbounded domain.  Section 6 contains a treatment of the functions of the second kind, where we show that they satisfy the same three-term recurrence relation of the Meixner--Pollaczek polynomials and have the same raising and lowering operators.

\section{The exponential function $E^{\lambda}(x,t)$}

\noindent The $T$-analogue of $e^{ixt}$  is
\bea\label{eq15}
E^{\lambda}(x, t) = g_\lambda(t) \sum_{n=0}^\infty \frac{(\lambda+(1-n)/2+ix)_n}{n!} \, t^n,
\eea
where $g_\lambda(t)$ is a constant chosen such that $E^{\lambda}(0, t)=1$.
The linearity of the operator $T$ and \eqref{eqTf} imply
\begin{align}\label{td}
TE^{\lambda}(x, t) =i\, t\,E^{\lambda}(x,t).
\end{align}

\begin{thm}
The $T$-exponential functions \eqref{eq15} is 
 \begin{equation}
E^{\lambda}(x, t) =\left(\frac{t}{2}+\sqrt{1+\frac{t^2}{4}}\right)^{2ix}=
\exp\left(2ix\, {arcsinh}\left(\frac{t}{2}\right)\right).
\end{equation}
\end{thm}
\begin{proof} We first simplify the series in the definition of $E^{\lambda}(x, t)$ by splitting it into even and odd indices.
It is clear that 
\begin{align*}
\sum_{n=0}^\infty \frac{(\lambda+(1-n)/2+ix)_n}{n!} \, t^n   =\sum_{n=0}^\infty \frac{(\lambda+\frac12-n+ix)_{2n}}{(2n)!} \, t^{2n}+\sum_{n=0}^\infty \frac{(\lambda-n+ix)_{2n+1}}{(2n+1)!} \, t^{2n+1}.
\end{align*}
Using the Pochhammer identities $(c-n)_{2n} = (1-c)_n(c)_n(-1)^n$, $\left(\frac{1}{2}\right)_n = \frac{(2n)!}{(4^n n!)}$, and $\left(\frac{3}{2}\right)_n = \frac{(2n+1)!}{(4^n n!)}$, and employing the series representation of hypergeometric functions, we conclude that
\begin{align*}
&\sum_{n=0}^\infty \frac{(\lambda+(1-n)/2+ix)_n}{n!} \, t^n
= {}_2F_1\left(\begin{array}{cc}\frac12+ \lambda+ix &\frac12-\lambda-ix \\   
\frac12  \end{array}\bigg|  -\frac{t^2}{4}\right) +t\, (\lambda+ix){}_2F_1
 \left(\begin{array}{cc} 1+ \lambda+ix &1-\lambda-ix \\   
\frac32 \end{array}\bigg|  -\frac{t^2}{4}\right). 
\end{align*} 
By means of the identities \cite[Formulas 7.3.2.90-91]{y1990}
\begin{align*}
{}_2F_1(a,1-a;\tfrac12;z)&=\frac{1}{\sqrt{1-z}}\cos((2a-1)\arcsin\sqrt{z}),\\
{}_2F_1(a,2-a;\tfrac32;z)&=\frac{1}{2(a-1)\sqrt{z-z^2}}\sin(2(a-1)\arcsin\sqrt{z}),
\end{align*}
we deduce then 
\begin{align*}
 \sum_{n=0}^\infty \frac{(\lambda+(1-n)/2+ix)_n}{n!} \, t^n&=\frac{\cos\left(2(\lambda+ix)\arcsin\left(\frac{it}{2}\right)\right)}{\sqrt{1+\frac{t^2}{4}}}-\frac{i\sin \left(2 (\lambda+i x) \arcsin\left(\frac{it}{2}\right)\right)}{ \sqrt{1+\frac{t^2}{4}}}\\
&=\frac{1}{\sqrt{1+\frac{t^2}{4}}}\exp\left(-2i(\lambda+ix)\arcsin\left(\frac{it}{2}\right) \right).
\end{align*}
 Thus we have shown that 
\begin{align*}
E^{\lambda}(x, t)&= \frac{g_\lambda(t) }{\sqrt{1+\frac{t^2}{4}}} \exp\left(2\,(\lambda+ix)\,\text{arcsinh}\left(\frac{t}{2}\right)\right)= \frac{g_\lambda(t) }{\sqrt{1+\frac{t^2}{4}}}\left(\frac{t}{2}+\sqrt{1+\frac{t^2}{4}}\right)^{2(\lambda+ix)}.
\end{align*}
We chose $g_\lambda(t)$ such that at $x=0$, $E^{(\lambda)}(0,t)=1$, hence
$$
g_\lambda(t)=\sqrt{1+\frac{t^2}{4}}\left(\frac{t}{2}+{\sqrt{1+\frac{t^2}{4}}}\right)^{-2 \lambda}.$$
\end{proof}
\noindent Note that the function $E^{\lambda}(x, t)$  is clearly independent of $\lambda$. Therefore from now on we shall use $E(x, t)$ instead of $E^{\lambda}(x, t)$.  The odd and even parts of $E(x,t)$ are  analogues of the sine and cosine 
functions,  so we set 
\bea
\begin{gathered}
C(x,t) =\cos\left(2x\, \text{arcsinh}\left(\frac{t}{2}\right)\right),\\
S(x,t) =\sin\left(2x\, \text{arcsinh}\left(\frac{t}{2}\right)\right),
\end{gathered}
\eea
respectively. Indeed $E(x,t) = C(x,y)+i\, S(x,t)$.  It may be of interest to note that
 \begin{equation}\label{lim}
E^{\lambda}\left(x, 2  \sinh\left(\tfrac{t}2\right)\right) =e^{i\,x\,t}.
\end{equation}

\section{A plane wave formula}
\noindent In this section, we give the expansion of a plane wave, $E(x,t)$, in a series of Meixner-Pollaczek  polynomials. The main  result is the determination of the coefficients $\{g_n(t,\lambda)\}$ in the expansion
\begin{align}\label{e1}
E(x,t) = \sum_{n=0}^\infty\, g_n(t, \lambda)\, P_n^{(\lambda)}(x;\phi). 
\end{align}
This is the analogue of the expansion of the plane wave $e^{ixt}$ in spherical harmonics. The derivation uses the connection relation 
\begin{align}\label{e56}
(\lambda^2+x^2)&{P}_n^{(\lambda+1)}(x;\phi) \notag\\
& =\frac{(n+2) (n+1)}{(2 \sin\phi)^2} \bigg[P_{n+2}^{(\lambda)}(x;\phi ) -\frac{2(2 \lambda +n+1)  \cos\phi
}{n+2}P_{n+1}^{(\lambda)}(x;\phi ) 
+\frac{(2 \lambda +n)(2\lambda +n+1)}{(n+2) (n+1)}P_{n}^{(\lambda)}(x;\phi )\bigg]
\end{align}
that we  prove it first.  By using the weight function for the  Meixner--Pollaczek polynomials \eqref{eq4}, the Christoffel formula \cite[Theorem 2.7.1]{i2005} implies 
\bea
\label{eqChr}
(\lambda^2+x^2){P}_n^{(\lambda+1)}(x;\phi) = C_n
 \left| 
\begin{array}{ccc}
  {P}_{n}^{(\lambda)}(-i \lambda;\phi )&  {P}_{n+1}^{(\lambda)}(-i \lambda;\phi )& {P}_{n+2}^{(\lambda)}(-i \lambda;\phi )\\ \\
  {P}_{n}^{(\lambda)}(i \lambda;\phi )&  {P}_{n+1}^{(\lambda)}(i \lambda;\phi )& {P}_{n+2}^{(\lambda)}(i \lambda;\phi )\\ \\
   {P}_{n}^{(\lambda)}(x;\phi )&  P_{n+1}^{(\lambda)}(x;\phi )& {P}_{n+2}^{(\lambda)}(x;\phi )
   \end{array}
\right|\eea
for some constant $C_n$. The generating function
\eqref{eq3} shows that $P_n(\pm i\lambda) = (2\lambda)_n e^{\pm in\f}/n!$.  On the other hand the three term recursion \cite[Formula (1.7.3)]{Koe:Swa}
\bea
\begin{gathered}
2[x\sin \f +(n+\lambda)\cos  \f] P_n^{(\lambda)}(x;\phi) 
= (n+1)P_{n+1}^{(\lambda)}(x;\phi)
+ (n+2\lambda-1)P_{n-1}^{(\lambda)}(x;\phi)
\end{gathered}
\label{eq3trr}
\eea
indicates that the coefficient of $x^n$ in $P_{n}^{(\lambda)}(x)$ is $(2\sin \f)^n/n!$. 
Therefore 
\bea
C_n  \left| 
\begin{array}{ccc}
  {P}_{n}^{(\lambda)}(-i \lambda;\phi )&  {P}_{n+1}^{(\lambda)}(-i \lambda;\phi )
  \\
  {P}_{n}^{(\lambda)}(i \lambda;\phi )&  {P}_{n+1}^{(\lambda)}(i \lambda;\phi )    \end{array}
\right| = \frac{(n+2)!}{4\, n! \sin^2\f}, 
\notag
\eea
that is 
\bea
\notag
C_n = -i \frac{(n+1)!(n+2)!}{(2\lambda)_n (2\lambda)_{n+1}(2\sin\f)^3}. 
\eea
Equation \eqref{e56}  then follows  from \eqref{eqChr}. 

\vskip0.1true in
 \noindent We now establish \eqref{e1}. It is clear that $E(x,t)$ belongs to the $L_2$ spaces weighted by $\omega(x;\l,\f)$ and $\omega(x;\l+1,\f)$.  Using \eqref{o5} 
and the fact that $T^2 E(x,t) = - t^2 E(x,t)$, which is a consequence of \eqref{td}, we find that 
\bea
-t^2 \sum_{n=0}^\infty\, g_n(t, \lambda)\, P_n^{(\lambda)}(x;\phi)
= (2 \sin\f)^2 \sum_{n=0}^\infty\, g_{n+2}(t, \lambda)\, P_n^{(\lambda+1)}(x;\phi).
\eea
Thus 
 \bea
 \notag
 \begin{gathered}
 \frac{2\pi\Gamma(n+2\l +2)}{n! (2\sin \f)^{2\l}} g_{n+2}(t, \lambda) 
 = -t^2    \sum_{m=0}^\infty\, g_m(t, \lambda)\, 
 \int_{-\infty}^\infty \omega(x;\l,\f)P_m^{(\lambda)}(x;\phi) (x^2+\l^2) P_n^{(\lambda+1)}(x;\phi) dx.
\end{gathered}
\eea
We then apply \eqref{e56} and the orthogonality relation \eqref{eqort} and conclude that 
\begin{align}\label{g2}
-t^2\left[g_n(t, \lambda)-2 \cos \phi\;   g_{n+1}(t, \lambda) 
+ g_{n+2}(t, \lambda)\right] = 4\sin^2\phi\, g_{n+2}(t, \lambda).
\end{align}
This equation can be viewed as a difference equation with constant coefficients, whose general solution is given by
\begin{align}
\label{eqgn}
g_n(t,\lambda) 
&={c_1} \left(\frac{t^2 \cos\phi-it\sin\phi\sqrt{t^2+4}}{2+t^2-2 \cos (2 \phi)}\right)^n
+{c_2}\left(\frac{t^2 \cos \phi+it\sin\phi\sqrt{t^2+4}}{2+t^2-2 \cos (2 \phi)}\right)^n.
\end{align}
To find the values of the constants $c_1$ and $c_2$, we need to compute $g_0$ and $g_1$. Recall the integral  evaluation \cite{m1978}:
\begin{equation}\label{eq16}
(\sec z)^\lambda=\frac{2^{\lambda-2}}{\pi\Gamma(\lambda)}\int_{-\infty}^\infty e^{zt}\left|\Gamma\left(\frac{\lambda+it}{2}\right)\right|^2dt,\quad\quad \lambda>0, \Re(z)\in\left(-\frac{\pi}{2},\frac{\pi}{2}\right). 
\end{equation}
It is clear from \eqref{e1} and \eqref{eqort} that 
\begin{align}
g_0(t,\lambda) &= \frac{(2\sin \f)^{2\l}}{2\pi\Gamma(2\l)} \int_{-\infty}^\infty \omega(x;\l,\f) E(x,t)
dx, \notag\\
g_1(t,\lambda)&= \frac{(2\sin \f)^{2\l}}{2\pi\Gamma(2\l+1)}\int_{-\infty}^\infty \omega(x;\l,\f)E(x,t)P_1^{(\lambda)}(x;\phi) dx\notag\\
 &= \frac{(2\sin \f)^{2\l}}{2\pi\Gamma(2\l+1)}\int_{-\infty}^\infty \omega(x;\l,\f)E(x,t)[2\lambda \cos \f + 2 x \sin \f] dx. 
\end{align}
We now evaluate $g_0$ and $g_1$ using the integral evaluation \eqref{eq16} and its derivative. The result is 
\bea
\begin{gathered}
g_0(t,\lambda)
=\left(\frac{\sin\phi}{\sin\left(\phi +\sinh^{-1}\frac{t}{2}\right)}\right) ^{2\lambda}, \\
g_1(t,\lambda)=\frac{i t }{2\sin\phi }\left(\frac{\sin \phi}{\sin\left(\phi +i \sinh ^{-1}\left(\frac{t}{2}\right)\right)}\right) ^{2 \lambda +1}.
\end{gathered}
\eea
By inspection we conclude that 
 \begin{align}
g_n(t, \lambda)=\left(\frac{i t}{2 \sin\phi}\right)^n\left(\frac{\sin\phi}{\sin\left (\phi+i\, \text{arcsinh}\left(\frac{t}{2}\right)\right)}\right)^{2\lambda+n},\quad n=0,1,\cdots.
\end{align}
It is easy to see that the above solution satisfies the difference equation and the initial conditions.  
\vskip0.1true in
\noindent In conclusion, we have established the expansion
\bea
\label{eqPW}
E(x,t) \equiv\exp\left(2ix\, \text{arcsinh}\left(\frac{t}{2}\right)\right) 
    = \Sum \left(\frac{i t}{2 \sin\phi}\right)^n\left(\frac{\sin\phi}{\sin\left (\phi+i\, \text{arcsinh}\left(\frac{t}{2}\right)\right)}\right)^{2\lambda+n}P_n^{(\lambda)}(x;\phi). 
\eea

\section{Sturm--Liouville problem}
\noindent Consider the real  inner product 
\bea
\label{eqinnpro}
(f, g) = \int_\R\, f(x)\, g(x)\, dx.
\eea
Here we assume that $f$ is real-valued and consider the corresponding {$L_2(\mathbb R,dx)$} space. Clearly, the operator $T$ defined by \eqref{t1} is densely defined on this space. 
\begin{thm}\label{thmA}
Let $f$ and $g$ be analytic in the strip $-1 \le \Im z \le 1$ and 
\bea
\lim_{x \to \pm \infty} \int_{-1/2}^{1/2} |f(x+iy+iu)|^2 dy=   
\lim_{x \to \pm \infty} \int_{-1/2}^{1/2} |g(x+iy+iu)|^2 dy  =0, 
\eea
uniformly for $|u| \leq 1/2$. 
Moreover, we assume that $f$ and $g$ are real-valued on $\R$. 
Then $$(Tf,g)=-(f, Tg),$$ that is $T^* = -T$. 
\end{thm} 
\begin{proof}It is clear by the definition of $T$ that
\begin{align}\label{eq6.1}
(Tf,g) &= -i \int_{-\infty}^\infty \left[f\left(x+\frac{i}2\right) - f\left(x-\frac{i}2\right)\right]g(x) dx\notag \\
&= -i \lim_{R\to \infty}\int_{-R}^R f\left(x+\frac{i}2\right) g(x)\, dx +i \lim_{R\to \infty}\int_{-R}^R f\left(x-\frac{i}2\right)\,g(x)\, dx.
\end{align}
With the change of variables, we can write the right-hand side of \eqref{eq6.1} as 
\begin{align}\label{eq6.2}
(Tf,g) =  -i \lim_{R\to \infty}\int_{-R+i/2}^{R+i/2} f\left(z\right) g(z-i/2) dz+ i \lim_{R\to \infty}\int_{-R-{i}/2}^{R-{i}/2} f\left(z\right)g(z+{i}/2) dz
=I_1+I_2.\end{align}
To evaluate this integral, we consider the following paths of integration:
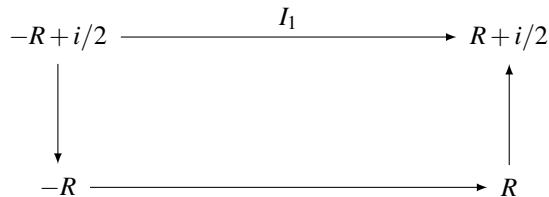
\begin{center}
\begin{tikzpicture}[inner sep=5pt, >=latex, 
   label/.style={auto, inner sep=1pt, circle}
   ]
\path (0,0) node (E) {$-R+i/2$}
(6,0) node (KK) {$R+i/2$}
(6,-2) node (K) {$R$}
(0,-2) node (F) {$-R$}
 (2,-2) node (EE) {\phantom{$E$}};
\draw[->] (E) -- node[label] { } (F);
\draw[->] (F) -- node[label] { } (K);
\draw[->] (K) -- node[label] { } (KK);
\draw[black,->] (E) -- node[label] {$I_1$} (KK);
\end{tikzpicture}
\end{center}
Hence we have 
\begin{align}\label{eq6.5}
I_1&=-i \lim_{R\to \infty}\int_{-R+i/2}^{-R} f\left(z\right) g\left(z-i/2\right) dz 
 -i \lim_{R\to \infty}\int_{-R}^{R} f\left(z\right) g\left(z-i/2\right) dz-i\lim_{R\to \infty}\int_{R}^{R+i/2} 
 f\left(z\right) g\left(z-i/2\right) dz.
\end{align}
While along a similar path of integrations, 
\begin{align}\label{eq6.6}
I_2&=i \lim_{R\to \infty}\int_{-R-i/2}^{-R} f\left(z\right)g\left(z+{i}/2\right) dz
+i \lim_{R\to \infty}\int_{-R}^{R} f\left(z\right)g\left(z+{i}/2\right) dz+i \lim_{R\to \infty}\int_{R}^{R-i/2} f\left(z\right)g\left(z+{i}/2\right) dz.
\end{align}
A further change of variables, equation \eqref{eq6.5} yields
\begin{align}\label{eq6.7}
I_1&= - \lim_{R\to -\infty}\int_0^{1/2} f\left(R+iy\right)\, g\left(R+iy-{i}/2\right) \,dy\notag\\
&-i \lim_{R\to \infty}\int_{-R}^{R} f\left(x\right)\, g\left(x-{i}/2\right) dx+\lim_{R\to \infty}\int_{0}^{1/2} f\left(R+iy\right) g\left(R+iy-{i}/2\right)dy.
\end{align}
While equation \eqref{eq6.6} may reads
\begin{align}\label{eq6.8}
I_2&=- \lim_{R\to -\infty}\int_{-1/2}^0 f\left(R+iy\right)\, g\left(R+iy+{i}/2\right) \,dy\notag\\
&+i \lim_{R\to \infty}\int_{-R}^{R} f\left(x\right)\, g\left(x+{i}/2\right) dx- \lim_{R\to \infty}\int_{0}^{-1/2} f\left(R+iy\right)g\left(R+iy+{i}/2\right) dy.
\end{align}
Therefore, from \eqref{eq6.7} and \eqref{eq6.8}, $i(Tf,g)$ equals 
\bea
\notag
\bg
\int_\R f(x)[g(x-i/2)-g(x+i/2)] dx \\
+ \lim_{R \to \infty} \left[\int_{1/2}^0 
f(-R +iy)g(-R+iy-i/2)i dy + \int_0^{1/2} f(R +iy)g(R+iy-i/2)i dy\right.\\
\left. - \int_{-1/2}^0 f(-R +iy)g(-R+iy+i/2)i dy - \int_0^{-1/2} 
f(R +iy)g(R+iy+i/2)i dy \right].
\eg
\eea
Therefore $(Tf,g) +(f,Tg)$ is given by 
\begin{align}\label{eq6.9}
(Tf,g) +(f,Tg)&=
 \lim_{R \to \infty}   \int_0^{1/2} \left[ f(R +iy)g(R+iy-i/2) 
+  f(R -iy)g(R-iy+i/2)\right] dy\notag \\
&- \lim_{R \to \infty} \int_0^{1/2} \left[ f(-R +iy)g(-R+iy-i/2) 
+  f(-R -iy)g(-R-iy+i/2)\right] dy.
\end{align}
It is clear that 
\bea
\notag
\bg
 \left|\int_0^{1/2}   f(R +iy)g(R+iy-i/2) dy\right|^2 
 \le  \int_0^{1/2}  \left|f(R +iy)\right|^2 dy \; 
 \int_0^{1/2} |g(R+iv-i/2)|^2 dv \to 0,  
 \eg
\eea
as $R\to \infty$. 
Similarly we estimate the remaining three integrals in equation \eqref{eq6.9}.
\end{proof}
 
\noindent We note that the identity $(Tf, g)=-(f,Tg)$ is the analogue of integration by parts.  
 \vskip0.1true in
 \noindent Now consider the following analogue of the Sturm--Liouville equation 
 \begin{align}\label{eq6.10}
 \frac{1}{\omega(x)} T [p(x)T y_n(x)] = \l_n y_n(x).
 \end{align}
  Here we assume that $\omega$ and $p$ are positive  on $\R$. 
 The following argument shows that the eigenfunctions corresponding to different eigenvalues are orthogonal with respect to $\omega$ on $\R$. Indeed 
 \begin{align}\label{eq6.11}
 (\l_m-\l_n) \int_\R \omega(x) y_m(x) y_n(x) dx 
 &= (TpTy_m, y_n)- (TpTy_n, y_m) =   -(pTy_m, T y_n) +(pTy_n,  T y_m) =0. 
 \end{align}
 
 \begin{thm}
 With the positivity of $p(x)$, the operator $- \frac{1}{\omega(x)} T [p(x)T]$ is a postive operator on the space of functions in {$L_2(\R,\omega)$}, which satisfy the assumption in Theorem \ref{thmA}.
 \end{thm}
 \begin{proof}  
 This follows from 
 \bea
 \left(- \frac{1}{\omega(x)} T [p(x)T]f, f\right)_{L_2} = -( T [p(x)T]f, f)= (pTf, Tf) \ge 0. 
 \notag
 \eea
 \end{proof}

  \section{Solutions to the Meixner--Pollaczek Recursion}
  
 \noindent  In this section we establish  a generating function for the general  solution to the recurrence relation \eqref{eq3trr}. We use the generating function to determine their large  $n$ behaviour via Darboux's asymptotic method \cite{a1984,Olv}.  We shall use the Stieltjes inversion formula: 
  \bea
  \label{eqStiInv}
  F(x) = \int_{-\infty}^\infty \frac{W(t)\,  dt}{x-t} \quad \textup{implies} \quad
  W(x) = \lim_{\epsilon \to 0+}  \frac{F(x-i\epsilon) - F(x+i\epsilon)}{2\pi i}. 
  \eea 
  We shall give an explicit representation of the numerator polynomials \cite[Section 2.3]{i2005} and give a new proof of the orthogonality relation of the Meixner--Pollaczek polynomials. 
  \vskip0.1true in
\noindent  Assume $y_0$ and $y_1$ are given and $\{y_n\}$  is generated by 
  \eqref{eq3trr} for $n >0$. Let $Y(t) = \Sum y_n t^n$. Multiply the recurrence relation by $t^n$ and add for $n =1,2, \cdots$. This yields 
  \bea
  \begin{gathered}
  [1-2 t\cos \f + t^2] \partial_t Y(t) + [2\l t -2x \sin \f -2\l \cos \f ]Y(t) = y_1 -2x \sin \f \, y_0 - 2\l \cos \f \, y_0.
 \end{gathered}
  \eea
  Thus  we find that 
  \bea
  \label{eqgf}
   \begin{gathered}
  (1-t e^{i\f})^{\l-ix}  (1-t e^{-i\f})^{\l+ix}  Y(t)  
  = y_0 +[y_1 -2x \sin \f \, y_0 - 2\l \cos \f \, y_0]\, \int_0^t (1-u e^{i\f})^{\l-ix-1}  
  (1-u e^{-i\f})^{\l+ix-1}\, du.
  \end{gathered}
   \eea
   
   \noindent We now study the asymptotics of these solutions. We start with the 
   Meixner--Pollaczek  polynomials. The generating function \eqref{eq3} has the $t$-singularities $t = e^{\pm i \f}$ and the dominant part of a comparison function at 
   these singularities are  $(1-te^{-i\phi})^{-\l-ix}(1- e^{2i\phi})^{-\l+ix}$ and  
   $(1-te^{i\phi})^{-\lambda+ix}(1- e^{-2i\phi})^{-\lambda-ix}$, respectively.  Therefore 
   \bea
   \label{eqAsym1}
 P_n^{(\l)}(x;\phi) \approx \frac{(\l+ix)_n}{n!} e^{-in\phi} (1- e^{2i\phi})^{-\l+ix}
 + \frac{(\l-ix)_n}{n!} e^{in\f} (1- e^{-2i\phi})^{-\l-ix},  
    \eea
    holds for all $x$ in the complex plane. 
  This shows that
  \bea
  \label{eqDarP}
 P_n^{(\l)}(x;\phi) \approx \frac{n^{\l-ix-1}}{\Gamma(\l-ix)}e^{in\f} (1- e^{-2i\f})^{-\l-ix},
 \eea
  when $\Im x > 0$.  The numerator polynomials $P_n^{(*)}(x, \l, \f)$ satisfy the  Meixner--Pollaczek recurrence relation with initial conditions $P_0^{(*)}(x, \l, \phi)=0, P_1^{(*)}(x, \l, \phi) = 2 \sin \phi$.  The generating function \eqref{eqgf} becomes 
 \bea
 \label{eqgf*}
  (1-t e^{i\f})^{\l-ix}  (1-t e^{-i\f})^{\l+ix}  \Sum P_n^{(*)}(x, \l, \f) t^n  
  = 2\sin \f \, \int_0^t (1-u e^{i\f})^{\l-ix-1}  
  (1-u e^{-i\f})^{\l+ix-1}\, du.
 \eea 
  We then apply Darboux's method \cite{a1984} and show that 
  \bea
  \label{eqStT}
  \lim_{n\to \infty} \frac{P_n^{(*)}(x, \l, \f)}{ P_n^{(\l)}(x;\phi) } = 2\sin \f \int_0^{e^{-i\f} } 
  (1-u e^{i\f})^{\l-ix-1}  
  (1-u e^{-i\f})^{\l+ix-1}\, du
  \eea
  holds for $\Im x > 0$.  It is clear from \eqref{eqDarP} that the orthonormal polynomials, say $\{p_n(x)\}$  do not belong to $L_2$ for $\Im x > 0$, hence the moment problem has a unique solution, see \cite{akh}.  Moreover the left-hand side of \eqref{eqStT} is the Stietjes transform of the orthogonality measure, which is normalized to have 
  total mass $=1$.  It is clear from  \eqref{eqStT}  that the Stietjes transform has no poles, hence the measure is absolutely continuous. This latter fact also 
  follows from \eqref{eqAsym1} because the series $\Sum |p_n(x)|^2$ diverges for all real $x$. 
  \vskip0.1true in
\noindent   Let $W$ be the normalized weight function. Thus the inversion formula \eqref{eqStiInv} gives 
  \bea
  \begin{gathered}
  W(x) = \lim_{\epsilon \to 0+} \frac{F(x-i\epsilon) - F(x+i\epsilon)}{2\pi i}= \frac{2\sin \f}{2\pi i}  \int_{e^{-i\f} }^ {e^{i\f} }
  (1-u e^{i\f})^{\l-ix-1}  
  (1-u e^{-i\f})^{\l+ix-1}\, du.
  \end{gathered} 
  \eea
  This is a beta integral and the change of variable $u = e^{-i\f} + (e^{i\f}-e^{-i\f}) v$ transforms it to the standard form. 
  Therefore the normalized weight function is given by 
  \bea
  W(x) = \frac{(2\sin \f)^{2\l}}{2\pi  \Gamma(2\lambda)} \; e^{(2\phi-\pi)x}\, |\Gamma(\lambda+ix)|^2. 
  \eea
  We note that the generating functions \eqref{eq3} and \eqref{eqgf*} prove the explicit representation 
  \bea
  P_n^{(*)}(x, \l, \f) =2 \sin \f  \sum_{k=0}^{n-1} \frac{P_k^{(\l)}(x; \f) \, P_{n-k-1}^{(1-\l)}(-x; \f)}{n-k}.
  \eea
  
  \section{Function of the second kind} 
  \noindent Define 
 the function of the second kind  $Q_n^{(\l)}(z;\phi)$ by   
 \bea
 \label{eqfn2ndkind}
\omega(z; \l, \f) Q_n^{(\l)}(z;\phi) = \int_\R \frac{P_n^{(\l)}(t;\phi)}{z-t} \omega(t; \l, \f) dt,  
 \eea
 defined for $\Im z \ne 0$. Because we have two variables, we shall use $T_z$ and $T_t$ to denote the action of the operator $T$ on a function of $z$ or $t$. It is clear, by the definition of $T$, that $$T_z\left(\frac{1}{z-t}\right) = - T_t\left(\frac{1}{z-t}\right).$$ 
 \begin{thm}
 The function of the second kind also satisfies \eqref{o3} and \eqref{eqraise}, that is 
 \bea\label{eq48}
 \begin{gathered}
 T Q_n^{(\l)}(z;\phi) = 2\sin \f \, Q_{n-1}^{(\l+1/2)}(z;\phi), \\
 T[ \omega(z; \l, \f) Q_n^{(\l)}(z;\phi)] =  -(n+1)\,  \omega(z; \l-\tfrac12, \f) \, Q_{n+1}^{(\l-1/2)}(z;\phi).
 \end{gathered}
 \eea
  \end{thm}
  \begin{proof} We first prove the second relation. Clearly 
 \begin{align*}
T_z[ \omega(z; \l, \f) Q_n^{(\l)}(z;\phi)]  &=  - \int_\R T_t\left[\frac{1}{z-t} \right] \omega(t; \l, \f) P_n^{(\l)}(t;\phi)dt  = \int_\R \frac{T_t[\omega(t; \l, \f) P_n^{(\l)}(t;\phi)]}{z-t} dt  \\
&= 
-(n+1)  \int_\R \frac{P_{n+1}^{(\l-1/2)}(t;\phi)}{z-t}\omega(t; \l-\tfrac12, \f) dt \\
&= -(n+1)\,  \omega(z; \l-\tfrac12, \f) \, Q_{n+1}^{(\l-1/2)}(z;\phi).
\end{align*}
To prove the lowering relation we note that
 \begin{align*}
2\sin \f \,  Q_{n-1}^{(\l+1/2)}(z;\phi) &= \frac{1}{\omega(z; \l + 1/2, \f)} 
\int_{-\infty}^\infty \frac{\omega(t; \l + 1/2, \f)} {z-t} T_t P_{n}^{(\l)}(t;\phi) \, dt\\
&= -  \frac{1}{\omega(z; \l + 1/2, \f)} 
\int_{-\infty}^\infty  P_{n}^{(\l)}(t)  T_t\left [\frac{\omega(t; \l + 1/2, \f)} {z-t} \right]\, dt,
 \end{align*}
where we used the analogue of integration by parts in the last step (see Section 4.). A calculation shows that  the integral in the above relation is 
\bea
\notag 
-\int_{-\infty}^\infty  P_{n}^{(\l)}(t;\phi) \omega(t; \l, \f)\left[\frac{e^{i \phi } (\lambda-it )}{z-\frac{i}{2}-t}+\frac{e^{-i \phi } (\lambda+it)}{z+\frac{i}{2}-t}\right]\, dt . 
\eea
Now write $\l - it$ as $\l  - iz -1/2 + i(z-t - i/2)$ and use the orthogonality of the $P_n$'s to see that 
\begin{align*}
 \frac{1}{\omega(z; \l + 1/2, \f)}&  \int_{-\infty}^\infty  P_{n}^{(\l)}(t;\phi) \omega(t; \l, \f) \frac{e^{i\f}(\l -it)}{z-t-i/2} dt\\
 &=  \frac{1}{\omega(z; \l + 1/2, \f)}  \int_{-\infty}^\infty  P_{n}^{(\l)}(t;\phi) \omega(t; \l, \f) \frac{e^{i\f}(\l  - iz -1/2)}{z-t-i/2} \, dt \\
&= \frac{i}{\omega(z-i/2; \l , \f)}  \int_{-\infty}^\infty    \frac{ P_n^{(\l)}(t;\phi)}{z-t-i/2} \omega(t; \l, \f) \, dt\\
&=i\,Q_n^{(\lambda)}(z-\frac{i}{2},\phi).
\end{align*}
Similarly we can handle the second integral to show that it is equal to $-i\,Q_n^{(\lambda)}(z+\frac{i}{2},\phi)$. 
\end{proof}

\noindent Next, we record the Rodrigues-type formula
\bea
\omega(x;\l,\f)Q_n^{(\lambda)} (x;\phi)= \frac{(-1)^n}{n!} T^n \left(\omega(x; \l+n/2, \f) Q_0^{(\lambda+n/2)}(x;\phi)\right)
\eea
that follows immediately from equation \eqref{eq48}. 
\vskip0.1true in
\noindent We note that \eqref{eqStT} implies that the Stieltjes transform of the normalized weight function $W$ is given by its right--hand side. Thus 
\bea
   \frac{(2\sin \f)^{2\l}} {\Gamma(2\l)} \int_{-\infty}^\infty 
   \frac{\omega(t;\l, \f)}{x-t} \, dt = 2\sin \f \int_0^{e^{-i\f} } 
  (1-u e^{i\f})^{\l-ix-1}  
  (1-u e^{-i\f})^{\l+ix-1}\, du,
  \eea
holds when $\Im x >0$.  This identifies $Q_0$ as
\bea
Q_0^{(\lambda)}(z;\phi) =  \frac{\Gamma(2\l)}{(2\sin \f)^{2\l-1}\omega(z;\l, \f)} \int_0^{e^{-i\f} } 
  (1-u e^{i\f})^{\l-iz-1}  
  (1-u e^{-i\f})^{\l+iz-1}\, du, 
\eea 
when $\Im z >0$, and $\overline{Q_0^{(\lambda)}(z;\phi)} = Q_0^{(\lambda)}(\bar{z};\phi)$ for all nonreal $z$.
\section{Acknowledgments and Funding}
\medskip
\noindent Partial financial support of this work under Grant No. RGPIN-2024-05913 from the
Natural Sciences and Engineering Research Council of Canada
 is gratefully acknowledged. The authors thank the anonymous referees for their insightful comments and suggestions, which have improved the manuscript.





\bibliographystyle{elsarticle-num}
\bibliography{<your-bib-database>}

\end{document}